\documentclass[11pt]{article}

\usepackage{amssymb,amsthm,amsmath,amssymb}
\usepackage{enumerate}
\usepackage{color}
\def\red{\color{red}}
\newcommand{\dd}{\mathrm{d}}
\newcommand{\E}{\mathbb{E}}
\newcommand{\1}{\textbf{1}}

\newcommand{\p}[1]{\mathbb{P}\left( #1 \right)}

\newcommand{\e}{\varepsilon}
\def\cD{{\mathcal D}}
\DeclareMathOperator{\Var}{Var}

\usepackage[paper=a4paper, left=1.4in, right=1.4in, top=1.2in, bottom=1.2in]{geometry}
\newtheorem{theorem}{Theorem}
\newtheorem{lemma}[theorem]{Lemma}

\theoremstyle{remark}
\newtheorem{remark}[theorem]{Remark}

\theoremstyle{definition}

\newcommand{\brac}[1]{\left(#1\right)}
\newcommand{\bfrac}[2]{\brac{\frac{#1}{#2}}}
\newcommand{\beq}[2]{\begin{equation}\label{#1}#2\end{equation}}
\def\cG{{\mathcal G}}
\newcommand{\set}[1]{\left\{#1\right\}}
\def\e{\varepsilon}
\begin{document}

\title{Random graphs with a fixed maximum degree}
\author{Alan Frieze\thanks{Research supported in part by NSF grant DMS1661063} and  Tomasz Tkocz\\Department of Mathematical Sciences\\Carnegie Mellon University\\Pittsburgh PA15217\\U.S.A.}
\maketitle
\begin{abstract}
We study the component structure of the random graph $G=G_{n,m,d}$. Here $d=O(1)$ and $G$ is sampled uniformly from $\cG_{n,m,d}$, the set of graphs with vertex set $[n]$, $m$ edges and maximum degree at most $d$. If $m=\mu n/2$ then we establish a threshold value $\mu_\star$ such that if $\mu<\mu_\star$ then w.h.p. the maximum component size is $O(\log n)$. If $\mu>\mu_\star$ then w.h.p. there is a unique giant component of order $n$ and the remaining components have size $O( \log n)$.
\end{abstract}

\maketitle

{\footnotesize

}
\bigskip

\section{Introduction}
We study the evolution of the component structure of the random graph $G_{n,m,d}$.  Here $d=O(1)$ and $G$ is sampled uniformly from $\cG_{n,m,d}$, the set of graphs with vertex set $[n]$, $m$ edges and maximum degree at most $d$. In the past the first author has studied properties of sparse random graphs with a lower bound on minimum degree, see for example \cite{Hamd3}. In this paper we study sparse random graphs with a bound on the maximum degree. The model we study is close to, but distinct from that studied by Alon, Benjamini and Stacey \cite{ABS} and Nachmias and Peres \cite{NP}. They studied the following model: begin with a random $d$-regular graph and then delete edges with probability $1-p$. They show in \cite{ABS} that for $d\geq 3$ there is a critical probability $p_c=\frac{1}{d-1}$ such that w.h.p. there is a ``double jump'' from components of maximum size $O(\log n)$ for $p<p_c$, a unique giant for $p>p_c$ and a mximum component size of order $n^{2/3}$ for $p=p_c$. The paper \cite{NP} does a detailed analysis of the scaling window around $p=p_c$.

Naively, one might think that this analysis covers $G_{n,m,d}$. We shall see however that $G_{n,m,d}$ and random subgraphs of random regular graphs have distinct degree sequence distributions. In the latter the number of vertices of degree $i=0,1,2,\ldots,d$ will be $n$ times a binomial random variable, whereas in $G_{n,m,d}$ this number will be asymptotic to $n$ times a Poisson random variable, truncated from above.

We will write that $A_n\approx B_n$ if $A_n=(1+o(1))B_n$ and $A_n\lesssim B_n$ if $A_n\leq (1+o(1))B_n$ as $n\to \infty$.

For $d \geq 1$ and $\lambda > 0$ define
\begin{equation}\label{eq:def-fk}
s_d(\lambda) = \sum_{j=0}^d \frac{\lambda^j}{j!}\qquad \text{ and }\qquad f_d(\lambda) = \lambda\frac{s_{d-1}(\lambda)}{s_{d}(\lambda)}.
\end{equation}
\begin{theorem}\label{thm:giants}
Let $d \geq 2$ and $\mu \in (0,d)$. Let $m = \lceil \frac{\mu n}{2}\rceil$. Let $G=G_{n,m,d}$ be a random graph chosen uniformly at random from the graphs with $n$ vertices, $m$ edges and maximum degree at most $d$. Let 
\[
\mu_\star(d) = f_d(f_{d-1}^{-1}(1)),\qquad\text{functional inverse being used here},
\] 
where the functions $f_k$ are defined in \eqref{eq:def-fk} and let $\lambda$ satisfy 
\begin{equation}\label{eq:deflam}
f_d(\lambda) = \mu.
\end{equation}
The following hold w.h.p. 
\begin{enumerate}[(a)]
\item The number $\nu_i,i=0,1,\ldots,d$ of vertices of degree $i$ in $G$ satisfies
\beq{degree}{
\nu_i\approx \lambda_in\text{ where }\lambda_i = \frac{1}{s_d(\lambda)}\frac{\lambda^i}{i!}.
}
\item If $\mu < \mu_\star(d)$, then $G$ has all components of size $O(\log n)$.
\item If $\mu > \mu_\star(d)$, then $G$ has a unique giant component of linear size $\Theta n$, where 
$\Theta$ is defined as follows:   let $D=\sum_{i=1}^Li\lambda_i$ and
\beq{deff}{
g(x)=D-2x-\sum_{i=1}^Li\lambda_i\brac{1-\frac{2x}{D}}^{i/2}.
}
Let $\psi$ be the smallest positive solution to $g(x)=0$. Then
\[
\Theta=1-\sum_{i=1}^L\lambda_i\brac{1-\frac{2\psi}{D}}^{i/2}.
\]
All the other components are of size $O(\log n)$.
\end{enumerate}
\end{theorem}

\begin{remark}\label{rem:mu*-explicit}
Numerical values of the threshold point $\mu_\star(d)$ for the average degree for small values of $d$ are gathered in Table \ref{tab:mu*}. Note that we have an exact expression for the case $d=3$. We use $f_2(\lambda)=\frac{\lambda(1+\lambda)}{1+\lambda+\lambda^2/2}$ to see that $f_2^{-1}(1)=\sqrt{2}$. And then $\mu_{\star} (3) =\frac{\lambda(1+\lambda+\lambda^2/2)}{1+\lambda+\lambda^2/2+\lambda^3/6}=3(\sqrt{2}-1)$.

Moreover, if we consider large $d$, then we have, as a function of $d$,
\begin{equation}\label{eq:mu*-large-d}
\mu_\star(d) = 1 + \frac{1}{e(d-1)!}-\frac{1}{ed!}+O\bfrac{1}{(d-1)!^2}.
\end{equation}
Comparing to the percolation model considered in \cite{ABS} and \cite{NP}, where $\mu_\star(d) = 1 + \frac{1}{d-1}$, we see that in our model a giant occurs \emph{ significantly earlier} for large $d$.
Approximation \eqref{eq:mu*-large-d} can be justified as follows. We have
\[
f_d(1) = \frac{s_{d-1}(1)}{s_d(1)} = 1 - \frac{1}{d!s_{d}(1)} = 1 - \frac{1}{e d!}+O\bfrac{1}{d!^2}
\] 
and 
\[
f_d'(1) = \frac{(s_{d-1}(1) + s_{d-2}(1))s_d(1)-s_{d-1}(1)^2}{s_d(1)^2}=1-\frac{1}{e d!}+O\bfrac{1}{d!^2},
\] 
(Express here $s_{d-1}$ and $s_{d-2}$ in terms of $s_d$ and use $s_d(1) =e - O(1/d!)$). 

If $f_{d-1}^{-1}(1) = 1 + \varepsilon$, then
\[
  1 = f_{d-1}(1+\varepsilon)= f_{d-1}(1) + f_{d-1}'(1)\varepsilon+O(\e^2),
\]
which gives 
\[
\varepsilon +O(\e^2)= \frac{1 - f_{d-1}(1)}{f_{d-1}'(1)}=\frac{1}{e(d-1)!} + O\bfrac{1}{(d-1)!d!}.
\] 
Consequently, 
\[
\mu_\star(d) = f_d(1 + \varepsilon)=f_d(1) + f_d'(1)\frac{1 - f_{d-1}(1)}{f_{d-1}'(1)}+O(\e^2). 
\]
and \eqref{eq:mu*-large-d} follows.
\end{remark}

\begin{table}[h!]\label{tab:mu*}
\begin{center}
\begin{tabular}{|c|c|}\hline
$d$ & $\mu_\star(d)$ \\\hline
$2$ & $\infty$ \\
$3$ & $3(\sqrt{2}-1)= 1.23264\ldots$ \\
$4$ & $1.05783$ \\
$5$ & $1.01309$ \\
$6$ & $1.00259$ \\
$7$ & $1.00044$ \\
$8$ & $1.00006$ \\\hline
\end{tabular}
\end{center}\caption{Numerical values of $\mu_\star(d)$ for small $d$.}
\end{table}

\section{Proof of Theorem \ref{thm:giants}}

The main idea is to estimate the degree distribution of $G_{n,m,d}$ and then apply the results of Molloy and Reed \cite{mol1}, \cite{mol2}.

\subsection{Technical Lemmas}

The following lemmas will be needed for the proof of part (a).
\begin{lemma}\label{lm:trun-poiss}
Let $\lambda > 0$, $d \geq 1$. Let $Z_1, Z_2, \ldots$ be i.i.d. random variables with 
\begin{equation}\label{eq:def-trun-poiss}
\p{Z_i = k} = c_\lambda\frac{\lambda^k}{k!}, \qquad k = 0,1,\ldots,d,
\end{equation}
where 
\begin{equation}\label{eq:def-sum}
c_\lambda=\frac{1}{s_d(\lambda)}.
\end{equation}  
(a truncated Poisson distribution). Let $(x_1,\ldots, x_n)$ be a random vector of occupancies of boxes when $m$ distinguishable balls are placed uniformly at random into $n$ labelled boxes, each with capacity $d$. Then the vector $(Z_1,\ldots,Z_n)$ conditioned on $\sum_{j=1}^n Z_j = m$ has the same distribution as $(x_1,\ldots, x_n)$.
\end{lemma}

\begin{proof}
Let $A$ be the set of vectors $z = (z_1,\ldots,z_n)$ of non-negative integers $z_j$ such that $\sum_{j=1}^n z_j = m$ and $z_j \leq d$ for every $j$. Fix $z \in A$. We have
\begin{align*}
\p{(Z_1,\ldots,Z_n) = z \ \Big| \ \sum_{j=1}^n Z_j = m} &= \frac{\p{(Z_1,\ldots,Z_n) = z}}{\p{\sum_{j=1}^n Z_j = m}} \\
&=
 \frac{\prod_{j=1}^n c_\lambda\frac{\lambda^{z_j}}{z_j!}}{\sum_{z \in A} \prod_{j=1}^n c_\lambda\frac{\lambda^{z_j}}{z_j!}} 
= \frac{\frac{1}{z_1!\cdot\ldots\cdot z_n!}}{\sum_{z \in A}\frac{1}{z_1!\cdot\ldots\cdot z_n!}}.
\end{align*}
On the other hand, there are $\frac{m!}{z_1!\cdot\ldots\cdot z_n!}$ ways to place $m$ balls into $n$ labelled boxes in such a way that the $j$th box gets $z_j$ balls. Therefore,
\[
\p{(x_1,\ldots,x_n) = z} = \frac{\frac{m!}{z_1!\cdot\ldots\cdot z_n!}}{\sum_{z \in A}\frac{m!}{z_1!\cdot\ldots\cdot z_n!}} = \p{(Z_1,\ldots,Z_n) = z \ \Big| \ \sum_{j=1}^n Z_j = m}.
\]
\end{proof}

\begin{remark}\label{rem:trun-poiss-gen}
The same argument can be adapted to different constraints for the occupancies of the boxes. In general, we can replace $k\in\{0,1,\ldots,d\}$ by $k\in I$ for some set of non-negative integers $I$. For example, instead of restricting the maximal occupancy, we can require a minimal occupancy (which has appeared in Lemma 4 in \cite{AFP}), or that the occupancy is even, etc. 
\end{remark}

A straightforward consequence of a standard i.i.d. case of the local central limit theorem (see, e.g. Theorem 3.5.2 in \cite{Dur}) is the following lemma which will help us get rid of the conditioning from Lemma \ref{lm:trun-poiss}.

\begin{lemma}\label{lm:local-CLT}
Let $\lambda > 0$, $d \geq 1$. Let $Z_1, Z_2, \ldots$ be i.i.d. truncated Poisson random variables defined by \eqref{eq:def-trun-poiss} and \eqref{eq:def-sum}. Then
\begin{equation}\label{eq:local-CLT}
\sup_{m = 0,1,2,\ldots}\sqrt{n}\left|\p{Z_1+\ldots+Z_n = m} - \frac{1}{\sqrt{2\pi n \sigma^2}}\exp\left\{-\frac{(m-\mu n)^2}{2n\sigma^2}\right\}\right| \xrightarrow[n\to\infty]{} 0,
\end{equation}
where $\mu = \E Z_1$ and $\sigma^2 = \Var(Z_1)$.
\end{lemma}

We shall also need two lemmas concerning the function $s_d$ from \eqref{eq:def-fk}. A function $f$ is log-concave if $\log f$ is concave.
\begin{lemma}\label{lm:log-conc-sd}
For every $\lambda > 0$, the sequence $(s_d(\lambda))_{d=0}^\infty$ defined by \eqref{eq:def-fk} is log-concave, that is $s_{d-1}(\lambda)s_{d+1}(\lambda) \leq s_d(\lambda)^2$, $d \geq 1$.
\end{lemma}
\begin{proof}
First note that the product of log-concave functions is log-concave. Integration by parts yields
\begin{equation}\label{eq:s_d-integral}
e^{-\lambda}s_d(\lambda) = \int_\lambda^\infty \frac{t^d}{d!}e^{-t}\dd t.
\end{equation}
Given this integral representation, the log-concavity of $(s_d(\lambda))_{d=0}^\infty$ follows from a more general result saying that if $f: (0,\infty)\to [0,\infty)$ is log-concave, then the function $(0,+\infty) \ni p \mapsto \int_0^\infty \frac{t^p}{\Gamma(p+1)}f(t)\dd t$ is also log-concave (apply to $f(t) = e^{-t}\1_{(\lambda,\infty)}(t)$). This result goes back to Borell's work \cite{Bor} (for this exact formulation see, e.g. Corollary 5.13 in \cite{Oliv} or Theorem 5 in \cite{NO} containing a direct proof). 
\end{proof}

\begin{remark}
The above theorem and proof uses two related notions of log-concavity. They are reconciled by the fact that if $f: (0,\infty)\to [0,\infty)$ is log-concave then the sequence $f(i),i=0,1,\ldots$ is also log-concave.
\end{remark}

\begin{lemma}\label{lm:fk-monot}
For every $k \geq 1$, the function $f_k$ is strictly increasing on $(0,\infty)$ and onto $(0,k)$. In particular, the functional inverse, $f_k^{-1}: (0,k) \to (0,\infty)$ is well-defined, also strictly increasing.
\end{lemma}
\begin{proof}
Fix $k \geq 1$ and consider $f_{k}$: rewriting \eqref{eq:s_d-integral} in terms of the upper incomplete gamma function $\Gamma(s,x) = \int_x^\infty t^{s-1}e^{-t} \dd t$, we have
\[
f_{k}(x) = k\frac{x\Gamma(k,x)}{\Gamma(k+1,x)}.
\]
Differentiating,
\[
\frac{\Gamma(k+1,x)^2}{k}\frac{\dd}{\dd x}f_{k+1}(x) = (\Gamma(k,x)-x^ke^{-x})\Gamma(k+1,x)+x^{k+1}e^{-x}\Gamma(k,x).
\]
Using $\Gamma(k+1,x) = k\Gamma(k,x)+x^ke^{-x}$ we can express the condition $\frac{\dd}{\dd x}f_{k+1}(x) > 0$ as a quadratic inequality for $\Gamma(k,x)$:
\[
k\Gamma(k,x)^2+x^{k}e^{-x}(x-k+1)\Gamma(k,x)-x^{2k}e^{-2x}>0,
\]
or
\[
\brac{\Gamma(k,x)+\frac{x^{k}e^{-x}(x-k+1)}{2k}}^2> \frac{x^{2k}e^{-2x}}{k}+\bfrac{x^{k}e^{-x}(x-k+1)}{2k}^2
\]
or
\begin{equation}\label{y}
\Gamma(k,x)>\frac{x^ke^{-x}}{2k}(\sqrt{(x-k+1)^2+4k}-(x-k+1)).
\end{equation}
Let $h(x)$ be the left hand side minus the right hand side of \eqref{y}. Clearly, $h(0) = (k-1)! > 0$. Moreover, using a standard asymptotic expansion 
\[
\Gamma(k,x) \approx x^{k-1}e^{-x}\left(1 + \frac{k-1}{x} + \frac{(k-1)(k-2)}{x^2}+ \ldots\right), \ \text{as } x \to \infty,
\]
we can check that $h(x) \approx x^{k-1}e^{-x}(\frac{1}{x^2} + \ldots)$, so $h(x) \to 0$ as $x \to \infty$. Thus to see that $h(x) > 0$ for $x > 0$, it suffices to check that $h'(x) < 0$ for $x > 0$. We have,
\begin{align*}
h'(x) &= -x^{k-1}e^{-x} - \frac{x^{k-1}e^{-x}}{2k}(k-x)\left(\frac{x-k+1}{\sqrt{(x-k+1)^2+4k}}-1\right) \\
&= -\frac{x^{k-1}e^{-x}}{2k\sqrt{(x-k+1)^2+4k}}\Big(2k\sqrt{(x-k+1)^2+4k} + (k-x)\big((x-k+1) \\
&\qquad\qquad\qquad\qquad\qquad\qquad\qquad\qquad\qquad\qquad\qquad\qquad - \sqrt{(x-k+1)^2+4k}\big) \Big) \\
&=-\frac{x^{k-1}e^{-x}}{2k\sqrt{(x-k+1)^2+4k}}\left((k+x)\sqrt{(x-k+1)^2+4k} + (k-x)(x-k+1)\right) ,
\end{align*}
so $h'(x) < 0$ is equivalent to
\[
(k+x)\sqrt{(x-k+1)^2+4k} > (x-k)(x-k+1).
\]
When $k-1 < x < k$, the right hand side is negative, so the inequality is clearly true. Otherwise, squaring it, we equivalently get
\[
(k+x)^2((x-k+1)^2+4k) > (x-k)^2(x-k+1)^2
\]
which is clearly true because $(k+x)^2 > (x-k)^2$ for $x > 0$.

It is clear from \eqref{eq:def-sum} and \eqref{eq:def-fk} that $f_k$ is a ratio of two polynomials, each of degree $k$ and $f_k(x) = \frac{\frac{x^{k}}{(k-1)!}+ \ldots}{\frac{x^{k}}{k!}+\ldots}$, so $f_k(x) \to k$ as $x \to \infty$. This combined with the monotonicity and $f_k(0) = 0$ justifies that $f_k$ is a bijection onto $(0,k)$.
\end{proof}

\subsection{Main elements of the proof}
Let $\mathcal{D}$ be the set of all sequences of nonnegative integers $x_1,\ldots,x_n \leq d$ such that $\sum x_i = 2m$ (possible degrees). For $x \in \mathcal{D}$, let $\mathcal{G}_{n,x}$ be the set of all simple graphs on vertex set $[n]$ such that vertex $i$ has degree $x_i,\,i=1,2,\ldots,n$. We study graphs in $\cG_{n,x}$ via the Configuration Model of Bollob\'as \cite{Conf}. We do this as follows: let $Z_x$ be the multi-set consisting of $x_i$ copies of $i$, for $i=1,2,\ldots,n$ and let $z=z_1,z_2,\ldots,z_{2m}$ be a random permutation of $Z_x$. We then define $\Gamma_z$ to be the (configuration) multigraph with vertex set $[n]$ and edges $\set{z_{2i-1},z_{2i}}$ for $i=1,2,\ldots,m$. It is a classical fact that conditional on being simple, $\Gamma_z$ is distributed as a uniform random member of $\cG_{n,x}$, see for example Section 11.1 of \cite{book}.

Let $\alpha_x = \frac{\sum_i x_i(x_i-1)}{2m}$. Note that $0 \leq \alpha_x \leq d$. It is known that
\[
|\mathcal{G}_{n,x}| \approx e^{-\alpha_x(\alpha_x+1)}\frac{(2m)!}{\prod_i x_i!}
\]
as $n \to \infty$ with the $o(1)$ term being uniform in $x$ (in fact, depending only on $\Delta = \max_i x_i$). Here the term $ e^{-\alpha_x(\alpha_x+1)}$ is the asymptotic probability that $\Gamma_z$ is simple. Therefore, for any $x \in \mathcal{D}$, we have
\[
\p{G_{n,m,d} \in \mathcal{G}_{n,x}} = \frac{|\mathcal{G}_{n,x}|}{\sum_{y \in \mathcal{D}} |\mathcal{G}_{n,y}|} \lesssim e^{d(d+1)}\frac{\frac{(2m)!}{\prod_i x_i!}}{\sum_{y \in \mathcal{D}} \frac{(2m)!}{\prod_i y_i!}},
\]
which by Lemma \ref{lm:trun-poiss} gives
\[
\p{G_{n,m,d} \in \mathcal{G}_{n,x}} \lesssim e^{d(d+1)}\p{Z = x \ \Big| \ \sum_i Z_i = 2m},
\]
where $Z_1,\ldots,Z_n$ are i.i.d. truncated Poisson random variables defined in \eqref{eq:def-trun-poiss}.

For any graph property $\mathcal{P}$, we thus have
\begin{align}
\notag\p{G_{n,m,d} \in \mathcal{P}} &= \sum_{x \in \mathcal{D}} \p{G_{n,m,d} \in \mathcal{P} \ | \ G_{n,m,d} \in \mathcal{G}_{n,x}}\p{G_{n,m,d} \in \mathcal{G}_{n,x}} \\
\notag&= \sum_{x \in \mathcal{D}} \p{G_{n,x} \in \mathcal{P}}\p{G_{n,m,d} \in \mathcal{G}_{n,x}} \\
&\lesssim e^{d(d+1)}\sum_{x \in \mathcal{D}} \p{G_{n,x} \in \mathcal{P}}\p{Z = x \ \Big| \ \sum_i Z_i = 2m},\label{eq:P-sum-over-degrees}
\end{align}
where $G_{n,x}$ denotes a random graph selected  uniformly at random from $\mathcal{G}_{n,x}$.

To handle the conditioning, we have chosen $\lambda$ so that $\mu = \E Z_1$, that is the value of $\lambda$ given by \eqref{eq:deflam}.

From Lemma \ref{lm:local-CLT} we get that for arbitrary $\delta > 0$, for sufficiently large $n$,
\[
\p{Z_1+\ldots+Z_n = 2m} \geq -\frac{\delta}{\sqrt{n}} + \frac{1}{\sqrt{2\pi n \sigma^2}}\exp\left\{-\frac{(2m-\mu n)^2}{2n\sigma^2}\right\}.
\]
Since $2m - \mu n = 2\lceil \frac{\mu n}{2}\rceil - \mu n \leq 2$ and $\sigma^2 = \Var(Z_1)$ depends only on $\lambda$ and $d$, hence only on $\mu$ and $d$, for sufficiently large $n$, the exponential factor is greater than, say $1/2$. Adjusting $\delta$ appropriately and using that $\sigma^2 \leq \mu$, in fact, \[
\Var(Z_1) = \E Z_1(Z_1-1) - (\E Z_1)^2 + \E Z_1 = \lambda^2\frac{s_{d-2}(\lambda)s_d(\lambda) - s_{d-1}(\lambda)^2}{s_d(\lambda)} + \E Z_1,
\] 
which by Lemma \ref{lm:log-conc-sd} is bounded by $\E Z_1 = \mu$, we get for sufficiently large $n$,
\begin{equation}\label{eq:local-lowbd}
\begin{split}
\p{Z_1+\ldots+Z_n = 2m} &\geq \frac{1}{10\sqrt{\mu n}}.
\end{split}
\end{equation}
Thus, for every $x \in \mathcal{D}$,
\begin{equation}\label{eq:cond-removed}
\p{Z = x \ \Big| \ \sum_i Z_i = 2m} \leq \frac{\p{Z = x}}{\p{ \sum_i Z_i = 2m}} \leq 10\sqrt{\mu n} \p{ Z = x}.
\end{equation}

The next step is to break the sum in \eqref{eq:P-sum-over-degrees} into likely and unlikely degree sequences. Note that $\E\sum_{j=1}^d \1_{\{Z_j = i\}} = n\p{Z_1 = i} = n \lambda_i$. By Hoeffding's inequality,
\[
\p{\left|\sum_{j=1}^n \1_{\{Z_j = i\}} - n \lambda_i\right| > \varepsilon n \lambda_i} \leq 2e^{-\varepsilon^2n\lambda_i/3}, \qquad \varepsilon > 0.
\]
Put $\varepsilon = n^{-1/3}\frac{1}{\max_i \lambda_i}$. The union bound yields
\beq{degrees}{
\p{\exists i \leq d \ \left|\sum_{j=1}^n \1_{\{Z_j = i\}} - n \lambda_i\right| > n^{2/3}} \leq 2d\exp\left\{- n^{1/3}\frac{\min_i \lambda_i}{3(\max_i \lambda_i)^2}\right\}.
}
This proves (a). It also shows that w.h.p. $n\lambda_i,i=0,1,\ldots,d$ asymptotically defines the degree distribution of $G_{n,m,d}$. Also, given that $x$ is chosen uniformly at random from $\cD$, we see that the distribution of $G_{n,x}$ in this case is the same as the distribution of the configuration model for the given degree sequence. 

To prove (b) and (c), we will use the Molloy-Reed criterion (see \cite{mol1},\cite{mol2} and Theorem 11.11 in \cite {book} for the exact formulation we shall use). First define
\[
\mathcal{A} = \left\{x=(x_1,\ldots,x_n) \in \mathcal{D}, \ \exists i \leq d \ \left|\sum_{j=1}^n \1_{\{x_j = i\}} - n \lambda_i\right| > n^{2/3}\right\}.
\]
Then, using \eqref{eq:cond-removed} and \eqref{degrees},
\begin{align*}
\sum_{x \in \mathcal{A}} \p{G_{n,x} \in \mathcal{P}}\p{Z = x \ \Big| \ \sum_i Z_i = 2m} &\leq 10\sqrt{\mu n}\sum_{x \in \mathcal{A}} \p{Z = x } \\
&= 10\sqrt{\mu n}\p{Z \in \mathcal{A}} \\ 
&\leq 20d\sqrt{\mu n}\exp\left\{-n^{1/3}\frac{\min_i \lambda_i}{3(\max_i \lambda_i)^2}\right\}.
\end{align*}

It remains to handle the typical terms $x \in \mathcal{D} \setminus \mathcal{A}$ in \eqref{eq:P-sum-over-degrees}. For such $x$, we now estimate $p_x = \p{G_{n,x} \in \mathcal{P}}$ in two cases: for $\mathcal{P}$ being the complement of (i) ``there are only small components'', and (ii) ``there is a giant'' depending on the behaviour of the degree sequences. 

Let $Q = \sum_{i=0}^d i(i-2)\lambda_i$. Note that by the definition of $\mathcal{A}$, for every $x \in \mathcal{D} \setminus \mathcal{A}$, the number of vertices in $G_{n,x}$ is $n\lambda_i + O(n^{2/3})$, so it is justified to use the Molloy-Reed criterion and we obtain that: if $Q < 0$, then $\max_x p_x \to 0$ in the case (i), and the same if $Q > 0$ in the case (ii). Finally note that
\[
Q = \lambda^2\frac{s_{d-2}(\lambda)}{s_{d}(\lambda)} - \lambda\frac{s_{d-1}(\lambda)}{s_{d}(\lambda)} = f_{d}(\lambda)(f_{d-1}(\lambda) - 1)
\]
and Lemma \ref{lm:fk-monot} together with the definition of $\lambda$, that is \eqref{eq:deflam}, finishes the proof. The expression for $\Theta$ is in \cite{mol2}. (One can also find a simplified proof of the Molloy-Reed results in \cite{book}, Theorem 11.11.)

\section{Conclusions}
We have found tight expressions for the degree sequence of $G_{n,m,d}$ and we have used the Molloy-Reed results to exploit them. In future work, we plan to study the scaling window around $Q$ close to zero. Hatami and Molloy \cite{HM} consider this case and their results show that we can expect a maximum component size close to $n^{2/3}$ in this case. They deal with a general degree sequence and perhaps we can prove tighter results for our specific case.

\end{document}